\documentclass{article}

\usepackage{amsmath}
\usepackage{mathtools}
\usepackage{amssymb}
\usepackage{amsthm}
\usepackage{amscd}
\usepackage[alphabetic,backrefs,lite]{amsrefs}
\usepackage{mathrsfs}
\usepackage[all,cmtip]{xy}
\usepackage{enumerate}
\usepackage{cleveref}

\newcommand{\C}{{\mathbb C}}
\newcommand{\F}{{\mathbb F}}

\newcommand{\PP}{{\mathbb P}}
\newcommand{\Q}{{\mathbb Q}}

\newcommand{\T}{{\mathbb T}}
\newcommand{\Z}{{\mathbb Z}}
\newcommand{\Qbar}{{\overline{\Q}}}

\newcommand{\kbar}{{\overline{k}}}
\newcommand{\Kbar}{{\overline{K}}}

\newcommand{\isom}{\simeq}

\DeclareMathOperator{\id}{id}
\DeclareMathOperator{\rank}{rank}

\DeclareMathOperator{\lcm}{lcm}

\DeclareMathOperator{\End}{End}

\DeclareMathOperator{\GL}{GL}

\newtheorem{theorem}{Theorem}[section]
\newtheorem{proposition}[theorem]{Proposition}
\newtheorem{lemma}[theorem]{Lemma}
\newtheorem{corollary}[theorem]{Corollary}
\newtheorem{conjecture}[theorem]{Conjecture}

\theoremstyle{definition}
\newtheorem{definition}[theorem]{Definition}
\newtheorem{example}[theorem]{Example}
\newtheorem*{acknowledgements}{Acknowledgements}

\theoremstyle{remark}

\crefname{theorem}{Theorem}{Theorems}
\crefname{proposition}{Proposition}{Propositions}
\crefname{lemma}{Lemma}{Lemmae}
\crefname{corollary}{Corollary}{Corollaries}
\crefname{conjecture}{Conjecture}{Conjectures}
\crefname{definition}{Definition}{Definitions}
\crefname{example}{Example}{Examples}
\crefname{equation}{the equation}{the equations}
\Crefname{equation}{The equation}{The equations}

\title{Determination of the modular Jacobian varieties $J_1(M,MN)$ with the Mordell--Weil rank zero}
\author{Koji Matsuda\thanks{The University of Tokyo}}
\date{}

\begin{document}

\maketitle

\begin{abstract}
In this paper, we determine all modular Jacobian varieties $J_1(M,MN)$ over the cyclotomic fields $\Q(\zeta_M)$ with the Mordell--Weil
rank zero assuming the Birch--Swinnerton-Dyer conjecture,
following the method of Derickx, Etropolski, van Hoeij, Morrow, and Zureick-Brown.
\end{abstract}

\section{Introduction}

The possible torsion groups of Mordell--Weil groups of elliptic curves over $\Q$ are completely classified
by Mazur \cite{Mazur}.
By generalizing Mazur's method, the possible torsion groups over quadratic fields are also classified
by Kenku and Momose \cite{KenkuMomose} and by Kamienny \cite{Kamienny}.
Very recently a corresponding theorem for cubic fields is proven \cite{DEvHMZB},
but the higher degree cases are still unknown.

These are proven by considering the corresponding modular curves.
More precisely, the existence of an elliptic curve with certain torsion points
is essentially equivalent to the existence of certain rational points of the modular curve.
Hence this kind of problem leads us to studying the rational points of modular curves.

On the other hand, in general the set of rational points of a curve is related to its Jacobian variety.
For example, if the Mordell--Weil rank of the Jacobian variety is zero,
then considering the Riemann--Roch spaces of divisors,
we can determine all rational points of the curve by finite steps, at least in theory.

This observation leads us to considering the problem of determining the curves $X_1(M,MN)$
(for the definition see below) whose Jacobian varieties have the Mordell--Weil ranks zero.
Following the method of \cite[Theorem 3.1]{DEvHMZB}, in this paper
we show the following (the cases of $M = 1$ and $2$ are taken from \cite[Theorem 3.1]{DEvHMZB}):

\begin{theorem}[Main theorem] \label{main_theorem}
The rank of $J_1(M,MN)(\Q(\zeta_M))$ is zero if  $(M,N)$ is in the following list:
\begin{enumerate}[$M=1,$]
\item $N \in \{1,\dots,36,38,\dots,42,44,\dots,52,54,55,56,59,60,62,64,66,\\
68,\dots,72,75,76,78,81,84,87,90,94,96,98,100,108,110,119,\\
120,132,140,150,168,180\},$
\item $N \in  \{ 1,\dots,21,24,\dots,27,30,33,35,42,45 \},$
\item $N \in  \{ 1,\dots,10,12,14,16,20 \},$
\item $N \in  \{ 1,\dots,6 \},$
\item $N \in  \{ 1,\dots,4,6 \},$
\item $N \in  \{ 1,\dots,5 \},$
\item $N \in  \{ 1,2 \},$
\item $N = 1$
\item $N = 1$
\item $N = 1$
\setcounter{enumi}{11}
\item $N = 1$.
\end{enumerate}
If the Birch--Swinnerton-Dyer conjecture is true, then the converse holds.
\end{theorem}

Note that in \cite{DEvHMZB} the authors claim that the converse of this theorem for $M = 1,2$ also holds unconditionally.
In its proof they use the converse of Kato's theorem \cite[Corollary 14.3]{Kato},
which is, according to a communication with the authors of the article, still an open problem,
and hence the proof is incomplete.

Using this theorem we classify elliptic curves with certain kind of rational torsion points over cyclotomic fields.

In the final section, by generalizing the proof of \cref{main_theorem},
we prove a statement about the rank of $J_1(M,MN)(\Q(\zeta_M))$,
and compute it concretely for some $(M,N)$ which are not on the list in \cref{main_theorem}.
Also, assuming the Birch--Swinnerton-Dyer conjecture, we give a lower bound of the rank of $J_0(N)(\Q)$.

\section{Preliminaries}

Throughout this paper, $N$ and $M$ stand for positive integers unless otherwise stated.

For a subgroup $\Gamma$ of $\GL_2(\Z/N)$,
let $X_\Gamma$ denote the modular curve over $\Z[1/N][\zeta_N]^{\det \Gamma}$ corresponding to $\Gamma$.
The space $X_\Gamma$ is a smooth proper scheme of relative dimension $1$ with geometrically connected fibers.
Let $J_\Gamma$ denote its Jacobian variety.

For particular subgroups
\begin{equation*}
\Gamma_0(N) = 
\left\{ \left( \begin{matrix}a & b \\ 0 & d \end{matrix} \right) \in \GL_2 (\Z/N) \right\},
\end{equation*}
and
\begin{equation*}
\Gamma_1(M,MN) = 
\left\{ \left( \begin{matrix}1 & b \\ 0 & d \end{matrix} \right) \in \GL_2 (\Z/MN) \bigg |
b \equiv 0, d \equiv 1 \mod M \right\}
\end{equation*}
respectively, we denote the curve $X_\Gamma$ by $X_0(N)$ and $X_1(M,MN)$ respectively,
and denote the curve $X_1(1,N)$ and $X_1(N,N)$ by $X_1(N)$ and by $X(N)$ respectively.
The curve $X_1(M,MN)$, which is defined over $\Z[1/MN][\zeta_M]$,
is the modular curve parametrizing elliptic curves and their independent two rational points of 
orders $M$ and $NM$.
Moreover we denote similarly for their Jacobian varieties.

There is a canonical morphism $X_1(N) \to X_0(N)$, and it has the automorphism group isomorphic to
$(\Z/N)^* / \{ \pm 1 \}$.

Moreover for a subgroup $\Delta$ of $(\Z/N)^* / \{ \pm 1 \}$,
we denote the curve $X_1(N)/\Delta$ by $X_\Delta$,
and $J_\Delta$ its Jacobian variety.
The curve $X_\Delta$ is isomorphic to $X_{\Gamma_\Delta}$ for the group
\begin{equation*}
\Gamma_\Delta = 
\left\{ \left( \begin{matrix}a & b \\ 0 & d \end{matrix} \right) \in \GL_2 (\Z/N) \bigg | a \mod \pm 1 \in \Delta  \right\}.
\end{equation*}

For a normalized eigenform $f$ of weight $2$ of level $\Gamma_1(N)$,
let $K_f$ be the number field generated by the Fourier coefficients of $f$,
and let $A_f$ denote the abelian variety associated to $f$.
It is of dimension $[K_f : \Q]$ and an order of $K_f$ acts on it (\cite[Theorem 1]{Shimura}).
Then for a prime number $l$, since the homology group $H_1(X_1(N)(\C), \Q)$ is free of rank $2$ over $\T \otimes_\Z \Q$,
we have that the Tate module $V_l(A_f) = T_l(A_f) \otimes_{\Z_l} \Q_l$ is free of rank $2$ over $K_f \otimes_\Q \Q_l$,
and by the construction, for a prime $p \ne l$ satisfying that $p \nmid N$,
the trace of a $p$-th arithmetic Frobenius on  $V_l(A_f)$ is $a_p(f)$.
Moreover, if $f$ is a newform, then the modular abelian variety $A_f$ is simple
and $\operatorname{End} A_f \otimes_\Z \Q = K_f$ (\cite[Corollary 4.2]{Ribet}).
Moreover, for a subgroup $\Delta$ of $(\Z/N)^* / \{ \pm 1 \}$,
the modular Jacobian variety $J_\Delta$ is isogenous to the product
\begin{equation} \label{eq:simple_decomposition}
\bigoplus_M \bigoplus_f A_f^{m_f},
\end{equation}
where $M$ runs over positive divisors of $N$,
$f$ runs over the set of Galois conjugacy classes of newforms of level $M$
whose characters $\epsilon$ satisfy that $\epsilon(a) = 1$ for all $a \in \Z/M$ with $a \mod \pm \in \Delta \mod M$,
and $m_f = \sigma_0(N/M)$ is the number of positive divisors of $N/M$ (\cite[Proposition 2.3]{Ribet}).

For the ranks of Mordell--Weil groups of abelian varieties, Birch and Swinnerton-Dyer conjectured the following:

\begin{conjecture}[Birch--Swinnerton-Dyer conjecture] \label{BSD}
Let $A$ be an abelian variety over $\Q$, and assume that the L-function $L(s)$ of $A$ has an analytic
continuation to $\C$.
Then the order of zero of $L(s)$ at $s = 1$ and the rank of the Mordell--Weil group $A(\Q)$ are the same.
\end{conjecture}

Kato (\cite[Corollary 14.3]{Kato}) proves one direction of its special case:
Namely, for a normalized eigenform $f$ of weight $2$ of level $\Gamma_1(N)$,
if the order of the zero of the L-function $L(f,s)$ at $s = 1$ is zero,
then the rank of $A_f(\Q)$ is zero.
In the proof of \cite[Theorem 3.1]{DEvHMZB} the authors use the converse,
which is, according to a communication with the authors of the article, still an open problem,
and hence the proof is incomplete.
For the conditional results in this paper, we assume its converse.

For an abelian extension $K/\Q$ with the Galois group $G$,
we identify the characters $\chi : G \to \Qbar^* $
with the Dirichlet characters corresponding to it.

For a Dirichlet character $\chi$ of conductor $M$ and for a modular form
$f = \sum_n a_n q^n$ of weight $k$ of level $\Gamma_1(N)$ with the character $\varphi$ of conductor $N'$,
we denote by $f_\chi$ the twist $\sum_n \chi(n) a_n q^n$ of the modular form $f$.
This is a modular form of weight $k$ of level $\Gamma_1(N'')$ with the character $\chi^2 \varphi$
of conductor dividing $\lcm(M, N')$,
where $N'' = \lcm(N, N'M, M^2)$.
(\cite[Proposition 3.1.]{AtkinLi})
Note that, by checking their Fourier coefficients, we have that
if a modular form $f$ is a normalized eigenform, then so is the twist $f_\chi$.

Let $f$ be a normalized eigenform of weight $2$ of level $\Gamma_1(N)$.
Then there exists a newform $g$ of level $M$ dividing $N$ such that $a_p(f) = a_p(g)$ for every prime $p$ not dividing $N/M$.
Such $g$ is unique by the multiplicity one theorem.
We call such a newform the newform associated with $f$, and denote it by $f^{\operatorname{new}}$.
In this case, since the number field $K_g$ is generated by $a_p(g)$ for every $p$ not dividing $N$,
the number field $K_f$ contains $K_g$.

Let $\Delta$ be the group
\begin{equation*}
\ker ( (\Z/M^2N)^* / \pm 1 \to (\Z/MN)^* / \pm 1).
\end{equation*}
Let $NF$ denote the set of newforms of weight $2$, level dividing $M^2 N$, and conductor dividing $MN$,
and let $\overline{NF}$ be the set of Galois conjugacy classes of $NF$.
Let $D$ be the group of Dirichlet characters modulo $M$.
These play a crucial role in this paper.

\begin{lemma} \label{decomposition_J1MN}
The modular Jacobian variety $J_\Delta$ is isogenous to the product
\begin{equation*}
\bigoplus_{G_\Q f \in \overline{NF}} A_f^{\sigma_0(M^2N/N_f)},
\end{equation*}
where $N_f$ is the level of $f$ and $\sigma_0(M^2N/N_f)$ is the number of positive divisors of $M^2N/N_f$ as in preliminaries.
\end{lemma}

\begin{proof}
We show the statement by showing that the space $S_2(\Gamma_\Delta)$ of the cusp forms of level $\Gamma_\Delta$ is
equal to the direct product $\oplus_\epsilon S_2(\Gamma_1(M^2N), \epsilon)$,
where $\epsilon$ runs over the Dirichlet characters of conductor dividing $MN$.
The space $S_2(\Gamma_\Delta)$ is a subspace of $S_2(\Gamma_1(M^2N))$,
and the later space is the direct product $\oplus S_2(\Gamma_1(M^2N), \epsilon)$,
where the sum is taken over all characters $\epsilon$ mod $M^2N$.
Now by the definition of $\Delta$, for $f \in S_2(\Gamma_1(M^2N), \epsilon)$,
we have that $f$ is of level $\Gamma_\Delta$ if and only if the character $\epsilon$ is of conductor dividing $MN$.
\end{proof}

\begin{lemma} \label{twist_of_NF}
Sending a pair $(\chi, f)$, where $\chi \in D$ and $f \in NF$, to $(f_\chi)^{\operatorname{new}}$
defines a well-defined action of the group $D$ on the set $NF$ from right.
\end{lemma}

\begin{proof}
What we need to show is the well-definedness of the action.
Let $\chi \in D$ and let $f \in NF$.
Then as we remarked in preliminaries, the twist $f_\chi$ is a normalized eigenform of level $\lcm(M^2N, MNM, M^2) = M^2N$
with the character of conductor dividing $\lcm(MN, M) = MN$.
Thus its associated newform $(f_\chi)^{\operatorname{new}}$ lies in $NF$.
This shows the well-definedness.
\end{proof}

\section{Proof of the main theorem}

In order to determine whether the Mordell--Weil rank of $J_1(M,MN)$ is zero or not, we use the following
corollary of the theorem of Kato:
This seems to be well-known, but for the luck of reference we give a proof.

\begin{lemma}
Let $L/K$ be a finite abelian extension of number fields with the Galois group $G$,
$F$ a number field containing all $\exp(G)$-th roots of unity,
and for every character $\chi : G \to \C^*$, let $A_\chi$ be an abelian variety over $K$
on which an order of $F$ acts.
Assume that there exists a prime $l$ such that the Galois modules
$V_l(A_\chi)$ and $\chi \otimes_{\Q_l} V_l(A_1)$ over $F \otimes_\Q \Q_l$ are isomorphic for all $\chi$.
Let $B$ be the Weil restriction of $(A_1)_L$ along with the extension $L/K$.
Then $B$ is isogenous to $\bigoplus_\chi A_\chi$
\end{lemma}

\begin{proof}
Let $A = A_1$ and $\Kbar$ the separable closure of $K$.
First we claim that the Galois module $V_l(B)$ is isomorphic to $V_l(A) \otimes_{\Q_l} \Q_l [G]$.
Since $L/K$ is Galois, we have that $\Kbar \otimes_K L \cong \Kbar[G]$.
Under this isomorphism, for $\sigma \in G_K$, the automorphism $\sigma \otimes \id$ on
$\Kbar \otimes_K L$ corresponds to the canonical action on $\Kbar[G]$.
Hence we have that
\begin{equation*}
B(\Kbar) = A(\Kbar \otimes _K L) = A(\prod_{\sigma \in G} \Kbar) = \bigoplus_{\sigma \in G} A(\Kbar)
= A(\Kbar) \otimes_\Z \Z[G].
\end{equation*}
Thus the claim follows.
Therefore since $G$ is abelian, we have that
\begin{equation*}
\begin{aligned}
V_l(B) & \isom V_l(A) \otimes_{\Q_l} \Q_l[G] \\
& \isom \bigoplus_\chi \chi \otimes V_l(A) \\
& \isom \bigoplus_\chi V_l(A_\chi).
\end{aligned}
\end{equation*}
Therefore by Faltings \cite[Korollar1]{Faltings} we have the result.
\end{proof}

\begin{corollary} \label{twist}
Let $f$ be a normalized eigenform of weight $2$ of level $\Gamma_1(N)$,
and $K/\Q$ be an abelian extension with the Galois group $G$.
Then we have
\begin{equation*}
\dim_{K_f} A_f(K) \otimes \Q = \sum_\chi \dim_{K_{f_\chi}} A_{f_\chi}(\Q) \otimes \Q,
\end{equation*}
where the sum is taken over all Dirichlet's characters of $K/\Q$.
\end{corollary}

\begin{proof}
Let $F$ be a finite extension of $K_f$ which contains all $\exp(G)$-th roots of unity.
The endomorphism ring of $A_{f_\chi}^{\oplus [F : K_{f_\chi}]}$ up to isogeny is isomorphic to
$M_{[F : K_{f_\chi}]} ( \End^0 A_{f_\chi} )$, which contains $M_{[F : K_{f_\chi}]} (K_{f_\chi})$.
Hence fixing a $K_{f_\chi}$-algebra homomorphism $F \to M_{[F : K_{f_\chi}]} (K_{f_\chi})$,
we let an order of $F$ act on $A_{f_\chi}^{\oplus [F : K_{f_\chi}]}$.
With this action, for every prime $l$, the Galois module $V_l(A_{f_\chi}^{\oplus [F : K_{f_\chi}]})$ is isomorphic to
$V_l(A_{f_\chi}) \otimes_{K_{f_\chi}} F$ over $F \otimes_\Q \Q_l$,
and for a prime $p$ not dividing $Nl$, a $p$-th arithmetic Frobenius $\varphi_p$ has the trace $\chi(p) a_p(f)$.
Therefore by the semi-simplicity of the Tate modules of abelian varieties (\cite[Satz3]{Faltings})
and by Chebotarev's density theorem, we have $V_l(A_{f_\chi}^{\oplus [F : K_{f_\chi}]}) \isom \chi \otimes V_l(A_f^{\oplus [F : K_f]})$.
Thus by the lemma we have $[F : K_f] \rank A_f(K) = \sum_\chi [F : K_{f_\chi}] \rank A_{f_\chi}(\Q)$.
\end{proof}

\begin{lemma} \label{new-eigen}
Let $f$ be a normalized eigenform of weight $2$ of level $\Gamma_1(N)$,
and let $g$ be the newform associated with $f$.
Then $A_f$ is isogenous to $A_g^{\oplus [K_f : K_g]}$,
and we have $\dim_{K_f}A_f(\Q) \otimes \Q = \dim_{K_g} A_g(\Q) \otimes \Q$.
\end{lemma}

\begin{proof}
Let $L = K_f$.
Then for a prime $l$ and for a prime $p$ not dividing $lN$,
the traces of a $p$-th arithmetic Frobenius on the Galois modules $V_l(A_f)$ and $V_l(A_g) \otimes_{K_g} L$
over $L \otimes_\Q \Q_l$ are $a_p(f) = a_p(g)$.
Therefore considering the decomposition of $L \otimes_\Q \Q_l$ into a product of fields and
corresponding decomposition of these modules, we have,
by semi-simplicity and by Chebotarev's density theorem, that these two modules are isomorphic.
In particular $V_l(A_f)$ and $V_l(A_g)^{\oplus [K_f : K_g]}$ are isomorphic over $\Q_l$.
Thus by Faltings' theorem we have that $A_f$ and $A_g^{\oplus [K_f : K_g]}$ are isogenous.
\end{proof}

Let $\Delta$ be the group
\begin{equation*}
\ker ( (\Z/M^2N)^* / \pm 1 \to (\Z/MN)^* / \pm 1).
\end{equation*}
Then by \cite{JeonKim}, we have that $(X_\Delta)_{\Q(\zeta_M)} \isom X_1(M,MN)$.
Recall that $NF$ is the set of newforms of weight $2$, level dividing $M^2 N$, and conductor dividing $MN$,
that $\overline{NF}$ is the set of Galois conjugacy classes of $NF$,
and that $D$ is the group of Dirichlet characters modulo $M$.

Using these statements we deduce the following crucial proposition.

\begin{proposition} \label{crucial}
Consider the following statements:
\begin{enumerate}[(1)]
\item The rank of $J_1(M,MN)(\Q(\zeta_M))$ is zero. \label{rankQzeta}
\item The rank of $J_\Delta(\Q)$ is zero. \label{rankQ}
\item For every newform $f \in NF$, the special value of the L-function $L(f,s)$ at $s=1$ is nonzero. \label{Lorder}
\end{enumerate}
Then we have $\eqref{rankQzeta} \iff \eqref{rankQ} \Leftarrow \eqref{Lorder}$.
Moreover if the Birch--Swinnerton-Dyer conjecture is true, then these statements are equivalent.
\end{proposition}

\begin{proof}
Since $(X_\Delta)_{\Q(\zeta_M)} \isom X_1(M,MN)$, the condition \eqref{rankQzeta} implies \eqref{rankQ}.

By \cref{decomposition_J1MN}, the modular Jacobian variety $J_\Delta$ is isogenous to
\begin{equation*}
\bigoplus_{G_\Q f \in \overline{NF}} A_f^{\sigma_0(M^2N/N_f)}.
\end{equation*}
Thus by the theorem of Kato, we have that \eqref{Lorder} implies \eqref{rankQ},
and that if the Birch--Swinnerton-Dyer conjecture is true, then these two statements are equivalent.

Finally by \cref{twist,decomposition_J1MN}, the rank of $J_1(M,MN)(\Q(\zeta_M)) = J_\Delta(\Q(\zeta_M))$ is zero
if and only if the rank of $A_{f_\chi} (\Q)$ is zero for every Dirichlet character $\chi$ modulo $M$ and
for every newform $f \in NF$.
Now such a twist is a normalized eigenform of level $\lcm(M^2N, MNM, M^2) = M^2N$
with the character of conductor dividing $\lcm(MN, M) = MN$.
Therefore by \cref{new-eigen} we have that \eqref{rankQ} implies \eqref{rankQzeta}.
\end{proof}

From this proposition, we easily prove the main theorem following the method of \cite{DEvHMZB}.

\begin{lemma} \label{rank0MN}
For $M \ge 3$, both of the ranks of $J_1(MN)(\Q)$ and of $J_0(M^2N)(\Q)$ are zero if $(M,N)$ is in the following list:
\begin{enumerate}[$M=1,$]
\setcounter{enumi}{2}
\item $N \in \{ 1, 2, 3, 4, 5, 6, 7, 8, 9, 10, 12, 14, 16, 20 \},$
\item $N \in  \{ 1, 2, 3, 4, 5, 6, 9 \},$
\item $N \in  \{ 1, 2, 3, 4, 6 \},$
\item $N \in  \{ 1, 2, 3, 4, 5 \},$
\item $N \in  \{ 1, 2 \},$
\item $N = 1,$
\item $N = 1,$
\item $N = 1,$
\setcounter{enumi}{11}
\item $N = 1.$
\end{enumerate}
Moreover if the Birch--Swinnerton-Dyer conjecture is true, then the converse holds.

For $M \ge 3$, the rank of $J_1(M^2N)(\Q)$ is zero if $(M,N)$ is in the following list:
\begin{enumerate}[$M=1,$]
\setcounter{enumi}{2}
\item $N \in \{ 1, 2, 3, 4, 5, 6, 8, 9, 10, 12, 20 \},$
\item $N \in  \{ 1, 2, 3, 4, 6 \},$
\item $N \in  \{ 1, 2, 3, 4, 6 \},$
\item $N \in  \{ 1, 2, 3, 5 \},$
\item $N \in  \{ 1, 2 \},$
\item $N = 1,$
\item $N = 1,$
\item $N = 1,$
\end{enumerate}
Moreover if the Birch--Swinnerton-Dyer conjecture is true, then the converse holds.
\end{lemma}

\begin{proof}
By \cite[Theorem 3.1]{DEvHMZB}.
As we remark in the introduction, the proof of \cite[Theorem 3.1]{DEvHMZB} seems to be invalid
since it requires the unproved part of the Birch--Swinnerton-Dyer conjecture, however the 'if' part is still valid.
\end{proof}

\begin{proof}[Proof of the main theorem]
First of all, for the group $\Delta$ as in the \cref{crucial}, there are morphisms
$X_1(M^2N) \to X_\Delta \to X_0(M^2N)$ over $\Q$
and $X_1(M,MN) \to X_1(MN)$ over $\Q(\zeta_M)$.
Hence using \cref{crucial}, we have that,
if the rank of $J_1(M^2N)(\Q)$ is zero then the rank of $J_1(M,MN)(\Q(\zeta_M))$ is zero,
and conversely if the rank of $J_1(M,MN)(\Q(\zeta_M))$ is zero then the ranks of $J_1(MN)(\Q)$ and of $J_0(M^2N)(\Q)$ are zero.

Thus by \cref{rank0MN} it remains to show that the special values of the L-functions of
$J_1(M,MN)(\Q(\zeta_M))$ at $s = 1$ are nonzero for
$(M,N) =(3,7),(3,14), \\ (3,16),(4,5),(6,4)$ and for $(12,1)$,
and that the special value of the L-function of $J_1(M,MN)(\Q(\zeta_M))$ for $(M,N) = (4,9)$ is zero.

Using the algorithm of \cite[Theorem 4.5.]{AgasheStein}, we can compute
for a newform $f$ whether $L(f,1)$ is zero or not.
(For example use Magma \cite{Magma}.)
According to it we have that $L(f,1) \neq 0$ for every newform of level dividing $M^2N$
with the character of conductor dividing $MN$, except the newforms
in the Galois conjugacy class of newforms labeled "G1N144I" by Magma \cite{Magma},
whose character has the conductor $36$,
and conversely for those newforms $f$, we have $L(f,1) = 0$.
Therefore by \cref{crucial} we have the desired result.
\end{proof}

Lastly from this theorem we deduce a corollary about the existence of certain kinds of elliptic curves.

\begin{corollary}
For $N=1,\dots,5$, there exist infinitely many elliptic curves over $\Q(\zeta_N)$ whose Mordell--Weil groups contain
subgroups isomorphic to $(\Z/N)^2$.
For $N=6,\dots,10,12$, there are no such elliptic curves over $\Q(\zeta_N)$.
For the other $N$, there are at most finitely many such elliptic curves over $\Q(\zeta_N)$.
\end{corollary}

\begin{proof}
Consider the modular curves $X(N)$ over $\Q(\zeta_N)$ which classifies generalized elliptic curves with their full level $N$ structures.
The third statement is just Faltings' theorem.
So assume $N=1,\dots,10$ or $12$.
For each $N$ we consider, by the theorem, the Mordell--Weil rank of its Jacobian variety is zero.
For $N=1,\dots,4$ or $5$, since the curve $X(N)$ has genus $0$ and of course has a rational point (the $\infty$-cusp),
it is isomorphic to $\PP^1$.
Therefore the result in this case is trivial.
Next assume $N=6,\dots,10$ or $12$.
For such an $N$, the curve $X(N)$ has genus greater than $0$ and is the fine moduli scheme of the corresponding stack.
Hence $X(N)(\Q(\zeta_N)) \to J(N)(\Q(\zeta_N))$ is injective.
By \cite[Appendix]{Katz},
for a good prime $\mathfrak{p}$ of the curve $X(N)$ over $\Q(\zeta_N)$ above $p$ satisfying $e(\mathfrak{p}/p) < p-1$,
since the Mordell--Weil rank is zero, the reduction map $J(N)(\Q(\zeta_N)) \to J(N)(\F_q)$ at $\mathfrak{p}$ is injective.
Note that the condition $e(\mathfrak{p}/p) < p-1$ is almost automatic:
the good primes of $X(N)/\Q(\zeta_N)$ are exactly the primes which do not divide the level $N$,
and the unramified primes of $\Q(\zeta_N)$ are exactly the primes which do not divide $N$,
except the characteristic $2$.
Thus in this case, by the diagram
\begin{equation*}
\begin{aligned}
\xymatrix{
X(N)(\Q(\zeta_N)) \ar@{^{(}->}[r] \ar[d] & J(N)(\Q(\zeta_N)) \ar@{^{(}->}[d] \\
X(N)(\F_q) \ar@{^{(}->}[r] & J(N)(\F_q),
}
\end{aligned}
\end{equation*}
we have that the reduction map $X(N)(\Q(\zeta_N)) \to X(N)(\F_q)$ is injective,
where $q$ is the norm of $\mathfrak{p}.$
On the other hand, by the Hasse bound, if $N^2 > (1 + \sqrt{q})^2$, i.e., if $q < (N-1)^2$, then the later set consists of cusps.
Since for a field $k$ with the characteristic prime to $N$,
the cusps of $X(N)(\kbar)$ correspond to the full level $N$ structures of the Neron $N$-gon over $\kbar$,
both of $X(N)(\Q(\zeta_N))$ and $X(N)(\F_q)$ contain all cusps.
Since of course the number of cusps of $X(N)(\kbar)$ is independent on $k$,
we have that if $q < (N-1)^2$ then the reduction map $X(N)(\Q(\zeta_N)) \to X(N)(\F_q)$ is surjective.
Hence now what we need to show is the claim that, for our $N$, there is a good prime $\mathfrak{p}$ of $\Q(\zeta_N)$
whose norm over $\Q$ is less than $(N-1)^2$ and whose characteristic is odd.
This is easy, and we have done.
\end{proof}

\section{Higher rank}

In the previous section we prove that the rank of $J_1(M,MN)(\Q(\zeta_M))$ is zero
if and only if the rank of $J_\Delta(\Q)$ is zero,
where $\Delta$ is the subgroup as in \cref{crucial}.
With more careful computation it is possible to show a more general statement.
Recall that $NF$ is the set of newforms of weight $2$, level dividing $M^2 N$, and conductor dividing $MN$,
and $\overline{NF} = G_\Q \backslash NF$ is the set of Galois conjugacy classes of $NF$.
Also recall that $D$ is the group of the Dirichlet characters modulo $M$.

\begin{proposition} \label{computing/Qzeta}
We have the following formula:
\begin{equation*}
\rank J_1(M,MN)(\Q(\zeta_M)) = \sum_{G_\Q f \in \overline{NF} } \sum_{\chi \in D}
\sigma_0(M^2N/N_{(f_\chi)^{\operatorname{new}}}) \rank A_f(\Q),
\end{equation*}
where $N_{(f_\chi)^{\operatorname{new}}}$ is the level of $(f_\chi)^{\operatorname{new}}$,
and $\sigma_0$ is as in the preliminaries.
\end{proposition}

Note that, decomposing $J_\Delta$ into simple factors as in \cref{decomposition_J1MN}, we have the following formula:
\begin{equation*}
\rank J_1(M,MN)(\Q(\zeta_M)) = \sum_{G_\Q f \in \overline{NF}} \sigma_0(M^2N/N_f) \rank A_f(\Q(\zeta_M)).
\end{equation*}
Thus using \cref{twist} we can compute the rank of $J_1(M,MN)(\Q(\zeta_M))$.
However, in order to compute $\rank J_1(M,MN)(\Q(\zeta_M))$ using this formula,
we need to compute the newform $g$ associated with the twist $f_\chi$ for every $\chi : (\Z/M)^* \to \C^*$
and for every $f$.
The importance of \cref{computing/Qzeta} is that
we only need to compute the newform associated with the twist $f_\chi$ only for $f$ satisfying
that $\rank A_f(\Q)$ is nonzero.

Before proving the proposition, we compute $\rank J_1(M,MN)(\Q(\zeta_M))$ for some $(M,N)$
using \cref{computing/Qzeta}.

\begin{example}[The rank of $J(11)(\Q(\zeta_{11}))$] \label{J(11)}
Let $\Delta$ be the subgroup as in \cref{crucial} for $(M,N) = (11,1)$.
According to \cite{LMFDB}, $J_\Delta$ has only one simple factor,
say $A$, whose rank over $\Q$ is nonzero.
($A = A_f$ for $f$ a newform in the Galois conjugacy class of newforms labeled "121.2.a.b" in \cite{LMFDB}.)
The newform $f$ has analytic rank $1$,
and the abelian variety $A$ is of dimension $1$.
Therefore by \cite[Corollary C]{Kolyvagin}, we have that $\rank A(\Q) = 1$.
Again according to \cite{LMFDB}, for every $\chi : (\Z/11)^* \to \C^*$,
the newform associated with $f_\chi$ has level $121$.
Therefore since there are $10$ characters $\chi : (\Z/11)^* \to \C^*$,
we have that $\rank J(11)(\Q(\zeta_{11})) = 10$.
\end{example}

\begin{example}[The rank of $J(14)(\Q(\zeta_{14}))$] \label{J(14)}
Let $\Delta$ be the subgroup as in \cref{crucial} for $(M,N) = (14,1)$.
According to \cite{LMFDB}, $J_\Delta$ has only one simple factor,
say $A$, whose rank over $\Q$ is nonzero.
($A = A_f$ for $f$ a newform in the Galois conjugacy class of newforms labeled "196.2.a.a" in \cite{LMFDB}.)
The newform $f$ has analytic rank $1$,
and the abelian variety $A$ is of dimension $1$.
Therefore by \cite[Corollary C]{Kolyvagin}, we have that $\rank A(\Q) = 1$.
Again according to \cite{LMFDB},
the newforms associated with $f_\chi$ have level $28$
for characters $\chi : (\Z/7)^* \to \C^*$ of order $6$
(there are $2$ such characters),
and the newforms associated with $f_\chi$ have level $196$
for other characters
(there are $4$ such characters).
Therefore we have that $\rank J(14)(\Q(\zeta_{14})) = 2\sigma_0(196/28) + 4\sigma_0(196/196) = 8$.
\end{example}

\begin{example}[The rank of $J(15)(\Q(\zeta_{15}))$] \label{J(15)}
Let $\Delta$ be the subgroup as in \cref{crucial} for $(M,N) = (15,1)$.
According to \cite{LMFDB}, $J_\Delta$ has two simple factors,
say $A_1$ and $A_2$, whose ranks over $\Q$ are nonzero.
($A_1 = A_{f_1}$ for $f_1$ a newform in the Galois conjugacy class of newforms labeled "225.2.a.a",
and $A_2 = A_{f_2}$ for $f_2$ in the class labeled "225.2.a.c" in \cite{LMFDB}.)
The newforms $f_i$ have analytic ranks $1$,
and the abelian varieties $A_i$ are of dimension $1$.
Therefore by \cite[Corollary C]{Kolyvagin}, we have that $\rank A_i(\Q) = 1$.
Again according to \cite{LMFDB},
for $f = f_1$,
the newform associated with $f_\chi$ has level $75$
for every character $\chi : (\Z/15)^* \to \C^*$ of conductor $3$ or $15$
(there are $4$ such characters),
and the newforms associated with $f_\chi$ have level $225$
for other characters
(there are $4$ such characters).
Next for $f = f_2$, a newform associated with $f_\chi$ has level $225$
for every $\chi : (\Z/15)^* \to \C^*$.
Therefore we have that $\rank J(15)(\Q(\zeta_{15})) = 4\sigma_0(225/75) + 12\sigma_0(225/225) = 20$.
\end{example}

For the subgroup $\Delta$ as in \cref{crucial} for $(M,N) = (13,1)$,
the modular Jacobian variety $J_\Delta$ has only one simple factor $A_f$ whose Mordell--Weil
rank over $\Q$ is nonzero.
The abelian variety $A_f$ has the dimension three and $f$ has the analytic rank one.
If Birch--Swinnerton-Dyer conjecture is true, then it follows that $\rank A_f(\Q) = 3$,
and thence we also can compute the rank of $J(13)(\Q(\zeta_{13}))$.
For a modular abelian variety $A_f$ if its L-function has a simple pole at $s = 1$,
then it seems that we obtain $\dim_{K_f} A_f (\Q) \otimes \Q = 1$ unconditionally.
We, however, could not find a reference.
Newforms of low levels have low analytic ranks.
For example, the newforms of level less than $389$ have analytic ranks at most $1$.
(See \cite{LMFDB}.)
Hence with this proposition we can easily compute the Mordell--Weil ranks of many modular Jacobian varieties.

\begin{proof}[Proof of \cref{computing/Qzeta}]
First by \cref{decomposition_J1MN} one obtains
\begin{equation*}
\rank J_1(M,MN)(\Q(\zeta_M)) = \sum_{G_\Q f \in \overline{NF}} \sigma_0(M^2N/N_f) \rank A_f(\Q(\zeta_M)).
\end{equation*}
For newforms $f$ and $g$, if these are Galois conjugate to each other,
then the fields $K_f$ and $K_g$ generated by their Fourier coefficients are isomorphic to each other,
and also the modular abelian varieties $A_f$ and $A_g$ are isogenous to each other.
Moreover for a newform $f$, the dimension of $A_f$ equals to $[K_f : \Q]$.
Therefore we have
\begin{equation} \label{eq:computing/Qzeta:1}
\rank J_1(M,MN)(\Q(\zeta_M)) = \sum_{f \in NF} \sigma_0(M^2N/N_f) \dim_{K_f} A_f(\Q(\zeta_M)) \otimes \Q.
\end{equation}
For each $f \in NF$, by \cref{twist,new-eigen} we have
\begin{equation*}
\dim_{K_f} A_f(\Q(\zeta_M)) \otimes \Q =
\sum_{\nu \in D} \dim_{K_{(f_\nu)^{\operatorname{new}}}} A_{(f_\nu)^{\operatorname{new}}}(\Q) \otimes \Q.
\end{equation*}
Thus combining them together we obtain
\begin{equation*}
\rank J_1(M,MN)(\Q(\zeta_M)) =
\sum_{f \in NF} \sum_{\nu \in D} \sigma_0(M^2N/N_f)
\dim_{K_{(f_\nu)^{\operatorname{new}}}} A_{(f_\nu)^{\operatorname{new}}}(\Q) \otimes \Q.
\end{equation*}
Since for every $\nu \in D$, the map
\begin{equation*}
\begin{aligned}
NF & \to NF \\
f  & \mapsto (f_\nu)^{\operatorname{new}}
\end{aligned}
\end{equation*}
is well-defined and bijective by \cref{twist_of_NF}, we obtain
\begin{equation*}
\rank J_1(M,MN)(\Q(\zeta_M)) =
\sum_{f \in NF} \sum_{\nu \in D} \sigma_0(M^2N/N_{(f_\nu)^{\operatorname{new}}}) \dim_{K_{f}} A_{f}(\Q) \otimes \Q.
\end{equation*}
Thus again by the same reason as in \cref{eq:computing/Qzeta:1},
the result follows.
\end{proof}

Using the same method as in \cite[Theorem 3.1, Remark 3.4]{DEvHMZB}, for fixed integer $r$,
we can get a necessary condition for that $\rank J_0(N)(\Q) = r$, assuming the Birch--Swinnerton-Dyer conjecture.
For a prime $p$, let $g_0^+(p)$ be the genus of the modular curve $X_0^+(p)$,
which is the quotient of the modular curve $X_0(p)$ by the Atkin--Lenner involution.

\begin{lemma} \label{rank_ge_genus}
If the Birch--Swinnerton-Dyer conjecture is true, then the following inequality holds:
\begin{equation*}
\rank J_0(p)(\Q) \ge g_0^+(p).
\end{equation*}
\end{lemma}

\begin{proof}
The modular Jacobian variety $J_0(p)$ is isogenous to $J^-_0(p) \times J_0^+(p)$.
Hence we have $\rank J_0(p)(\Q) \ge \rank J_0^+(p)(\Q)$.
The abelian variety $J_0^+(p)$ is the Jacobian variety of the modular curve $X_0^+(p)$,
and is isogenous to a product $\oplus_f A_f$, where the direct sum is taken over all Galois conjugacy classes of the newforms
of level $\Gamma_0(p)$ fixed by the Atkin-Lehner involution.
Thus for every simple factor of $J_0^+(p)$, its analytic rank is odd, and in particular is nonzero.
Hence if the Birch--Swinnerton-Dyer conjecture is true, then every simple factor of $J_0^+(p)$ has the nonzero Mordell--Weil rank.
Moreover, for every such newform $f$, since an order of $K_f$ acts on $A_f$,
we have $\rank A_f(\Q) \ge [K_f : \Q]$, which is equal to the dimension of $A_f$.
Thus we obtain
\begin{equation*}
\begin{aligned}
\rank J_0^+(p) (\Q) & = \sum_f \rank A_f(\Q) \\
& \ge \sum_f \dim A_f \\
& = \dim J_0^+(p),
\end{aligned}
\end{equation*}
where the sum is taken over all Galois conjugacy classes of the newforms of level $\Gamma_0(p)$ fixed by the Atkin-Lehner involution.
Thus the result.
\end{proof}

Note that \cref{rank_ge_genus} does not hold for composite numbers,
for example $\rank J_0(28)(\Q) = 0$ but $g_0^+(28) = 1$.
We also note that we know the complete list of prime numbers $p$
so that the genera $g_0^+(p)$ are less than $7$, see \cite[Proposition 4.5]{AABCCKW}.

\begin{lemma} \label{lem:strictly<}
Let $N$ be a positive integer.
\begin{enumerate}
\item Write $N = Mp^e$ for a positive integer $M$, for a prime $p$, and for $e \ge 1$.
For a newform $f$ of level $N_f$ dividing $M$, let $m_f$ be the prime-to-$p$-part of $M/N_f$.
Then $J_0(N)$ contains $J_0(M) \oplus (\oplus_f A_f^{e \sigma_0(m_f)})$ up to isogeny,
where $f$ runs over the Galois conjugacy classes of the newforms of level dividing $M$. \label{lem:strictly<:item1}
\item Assume that we can write $N = M_1 M_2$ for relatively prime positive integers.
Then $J_0(N)$ contains $J_0(M_1)^{\sigma_0(M_2)} \oplus J_0(M_2)^{\sigma_0(M_1)}$ up to isogeny. \label{lem:strictly<:item2}
\item Let $M$ be a positive proper divisor of $N$.
If the rank of $J_0(N)(\Q)$ is nonzero, then $\rank J_0(M)(\Q) < \rank J_0(N)(\Q)$. \label{lem:strictly<:item3}
\end{enumerate}
\end{lemma}

\begin{proof}
\begin{enumerate}
\item
First we show the first statement.
Let $M$ be a positive integer, $p$ a prime, $e \ge 1$ an integer, and $N = Mp^e$.
The modular Jacobian variety $J_0(M)$ decompose as $\oplus_f A_f^{\sigma_0(M/N_f)}$,
where $f$ runs over the Galois conjugacy classes of the newforms of level dividing $M$.
Since $M$ divides $N$, for each such $f$, the modular Jacobian variety $J_0(N)$ also contains $A_f$ up to isogeny,
with the multiplicity $\sigma_0(N/N_f)$.
Thus it suffices to show, for each $f$ as above and for $m_f$ as in the statement,
that $\sigma_0(N/N_f) = \sigma_0(M/N_f) + e \sigma_0(m_f)$,
which is elementary:
Write $M/N_f = p^b m_f$ for a nonnegative integer $b$.
Then since $m_f$ is prime-to-$p$, we obtain $\sigma_0(N/N_f) = \sigma_0(p^{e+b}) \sigma_0(m_f) = (e+b+1) \sigma_0(m_f)
= e \sigma_0(m_f) + (b+1) \sigma_0(m_f) = e \sigma_0(m_f) + \sigma_0(M/N_f)$.

\item
For the second statement let $M_i$ be positive integers as in the statement.
Since $M_1$ and $M_2$ are relatively prime to each other, considering the simple decomposition of $J_0(N)$,
we obtain that $J_0(N)$ contains $(\oplus_f A_f^{\sigma_0(N/N_f)}) \oplus (\oplus_g A_g^{\sigma_0(N/N_g)})$ up to isogeny,
where $f$ runs over the Galois conjugacy classes of the newforms of level dividing $M_1$,
and where $g$ runs over those of level dividing $M_2$.
Thus again since $M_i$ are relatively prime to each other, we obtain $\sigma_0(N/N_f) = \sigma_0(M_1 / N_f) \sigma_0(M_2)$
and $\sigma_0(N/N_g) = \sigma_0(M_2 / N_g) \sigma_0(M_1)$.
Since $\sigma_0(M_1 / N_f)$ (and $\sigma_0(M_2 / N_g)$) is the multiplicity of the modular abelian variety $A_f$
(and $A_g$) as a simple factor of $J_0(M_1)$ (and $J_0(M_2)$ respectively), the result follows.

\item
For the third statement, let $M$ be a positive proper divisor of $N$,
and assume that the rank of $J_0(N)(\Q)$ is nonzero.
If $\rank J_0(M)(\Q) = 0$, then there is nothing to show.
Assume that $\rank J_0(M)(\Q) > 0$.
In this case there exists a simple factor $A$ of $J_0(M)$ whose Mordell--Weil rank is nonzero.
By the first statement, for each simple factor $A$ of $J_0(M)$, we have that
$J_0(M) \times A$ is contained in $J_0(N)$ up to isogeny.
Thus the result.
\end{enumerate}
\end{proof}

\begin{proposition} \label{rank_tends_infinity}
Assume that the Birch--Swinnerton-Dyer conjecture is true.
Let $r$ be an integer.
Then there exist only finitely many integers $N$ such that $\rank J_0(N)(\Q) = r$,
i.e., the rank of $J_0(N)(\Q)$ tends to infinity as the level $N$ tends to infinity.
\end{proposition}

\begin{proof}
We show it by induction on $r$.
First by \cite[Theorem 3.1]{DEvHMZB}, the statement for $r = 0$ is true.
Next assume that $r > 0$ and suppose that there exist only finitely many integers $N$ such that $\rank J_0(N)(\Q) < r$.
Define
\begin{equation*}
h(x) = \frac{x - 5 \sqrt{x} + 4}{24} - \frac{\sqrt{x}}{\pi} (\log (16x) + 2).
\end{equation*}
We can compute that the function $h(x)$ is monotonically increasing for sufficiently large $x$, for example
for $x > 10^4$.
Since $h(x)$ tends to infinity as $x$ tends to infinity, we have that there exist only finitely integers $N$
such that $h(N) < r$.
Therefore by \cref{rank_ge_genus} and by \cite[Proposition 4.4]{AABCCKW},
there exist only finitely many primes $p$ such that $\rank J_0(p)(\Q) = r$.
By \cref{lem:strictly<} \eqref{lem:strictly<:item3}, if $N$ is a composite number then $N$ is of the form of $M_1 M_2$
for positive integers $M_i > 1$ satisfying that $\rank J_0(M_i)(\Q) < r$.
Therefore the induction hypothesis yields the result.
\end{proof}

Moreover the proof above inductively constructs, for each integer $r$, a finite set $S_r$ such that
if the rank of $J_0(N)(\Q)$ is equal to $r$ then $N \in S_r$.
Namely:

\begin{definition}
Let $S_0$ be the set of positive integers $N$ such that the rank of $J_0(N)(\Q)$ is zero
(\cite[Theorem 3.1]{DEvHMZB}).
For $r \ge 1$, assuming that we have defined the set $S_{r'}$ for every $r' < r$,
we define $S_r$ to be the set of positive integers $N$ satisfying that,
$N$ is either a prime such that $g_0^+(N) \le r$, or a composite number such that
every positive proper divisor $M$ of $N$ lies in the set $\cup_{i = 0}^{r-1} S_i$.
\end{definition}

For example we obtain that
\begin{equation*}
\begin{aligned}
S_0 = & \{ 1, 2, 3, 4, 5, 6, 7, 8, 9, 10, 11, 12, 13, 14, 15, 16, 17, 18, 19, 20, 21, 22, \\
& 23, 24, 25, 26, 27, 28, 29, 30, 31, 32, 33, 34, 35, 36, 38, 39, 40, 41, 42, 44, \\
& 45, 46, 47, 48, 49, 50, 51, 52, 54, 55, 56, 59, 60, 62, 63, 64, 66, 68, 69, 70, \\
& 71, 72, 75, 76, 78, 80, 81, 84, 87, 90, 94, 95, 96, 98, 100, 104, 105, 108, 110, \\
& 119, 120, 126, 132, 140, 144, 150, 168, 180 \}
\end{aligned}
\end{equation*}
by \cite[Theorem 3.1]{DEvHMZB}, and hence by definition we obtain
\begin{equation*}
\begin{aligned}
S_1 = & \{ 2, 3, 4, 5, 6, 7, 8, 9, 10, 11, 12, 13, 14, 15, 16, 17, 18, 19, 20, 21, 22, \\
& 23, 24, 25, 26, 27, 28, 29, 30, 31, 32, 33, 34, 35, 36, 37, 38, 39, 40, 41, 42, \\
& 43, 44, 45, 46, 47, 48, 49, 50, 51, 52, 53, 54, 55, 56, 57, 58, 59, 60, 61, 62, \\
& 63, 64, 65, 66, 68, 69, 70, 71, 72, 75, 76, 77, 78, 79, 80, 81, 82, 83, 84, 85, \\
& 87, 88, 89, 90, 91, 92, 93, 94, 95, 96, 98, 99, 100, 101, 102, 104, 105, 108, \\
& 110, 112, 115, 117, 118, 119, 120, 121, 123, 124, 125, 126, 128, 131, 132, 133, \\
& 135, 136, 138, 140, 141, 142, 143, 144, 145, 147, 150, 152, 153, 155, 156, 160, \\
& 161, 162, 165, 168, 169, 175, 177, 180, 187, 188, 189, 190, 192, 196, 200, 203, \\
& 205, 207, 208, 209, 210, 213, 216, 217, 220, 221, 225, 235, 238, 240, 243, 245, \\
& 247, 252, 253, 261, 275, 280, 287, 288, 289, 295, 299, 300, 315, 319, 323, 329, \\
& 341, 343, 355, 357, 360, 361, 377, 391, 403, 413, 437, 451, 475, 493, 497, 517, \\
& 527, 529, 533, 551, 589, 611, 649, 667, 697, 713, 767, 779, 781, 799, 833, 841, \\
& 893, 899, 923, 943, 961, 1003, 1081, 1121, 1189, 1207, 1271, 1349, 1357, 1363, \\
& 1457, 1633, 1681, 1711, 1829, 1927, 2059, 2201, 2209, 2419, 2773, 2911, 3337, \\
& 3481, 4189, 5041 \}
\end{aligned}
\end{equation*}

From statements above we can obtain a rough but easy-to-understand necessary condition for that $\rank J_0(N)(\Q) \le r$,
in the form of a lower bound of $\rank J_0(N)(\Q)$.

\begin{proposition} \label{lower_bound1}
Assume that the Birch--Swinnerton-Dyer conjecture is true.
Then $\rank J_0(N)(\Q) > -1 + \log_{180} N$.
\end{proposition}

\begin{proof}
We show the statement by showing that,
for a nonnegative integer $r$, if $\rank J_0(N)(\Q) \le r$, then we obtain $N \le 180^{r+1}$.
First we show a sharper result for prime numbers:
Namely, for a prime number $p$ and for a positive integer $r$, if $\rank J_0(p)(\Q) \le r$ then $p \le 180^r$,
and if $\rank J_0(p)(\Q) = 0$ then $p \le 180$.
Let $r$ be a nonnegative integer and $p$ a prime.
Define
\begin{equation*}
h(x) = \frac{x - 5 \sqrt{x} + 4}{24} - \frac{\sqrt{x}}{\pi} (\log (16x) + 2).
\end{equation*}
Let $f(x) = h(180^x) - x$.
Then we can compute that $f(x)$ is monotonically increasing for $x \ge 2$, and that $f(7) > 0$.
Hence for $r \ge 7$, if a prime $p$ satisfies $p > 180^r$ then $r < h(180^r) < h(p)$.
Therefore by \cref{rank_ge_genus} and by \cite[Proposition 4.4]{AABCCKW}, it shows our claim for $r \ge 7$.
For $1 \le r \le 6$, by \cite[Proposition 4.5]{AABCCKW} we have that if $g_0^+(p) \le r$ then $p \le 180^r$,
thence our claim for $1 \le r \le 6$.
Finally if $\rank J_0(p)(\Q) = 0$ then by \cite[Theorem 3.1]{DEvHMZB} we have $p \le 180$.

We show the statement for general case by induction on $r$.
Let $N$ be an integer so that $\rank J_0(N)(\Q) \le r$.
We may assume that $N$ is a composite number.
In the case $r = 0$, the statement follows from \cite[Theorem 3.1]{DEvHMZB}.
For $r = 1$, we computed $S_1$ explicitly above, hence the result. 
Let $r \ge 2$ and suppose that the statement holds for every $r' < r$.
By the induction hypothesis, we may assume that $\rank J_0(N)(\Q) = r$.

First assume that we can write $N = M_1 M_2$ for relatively prime integers $M_i > 1$.
Then by \cref{lem:strictly<} \eqref{lem:strictly<:item2} we have that $2 \rank J_0(M_1)(\Q) + 2 \rank J_0(M_2)(\Q) \le \rank J_0(N)(\Q)$.
Thus for each $i$ the rank of $J_0(M_i)(\Q)$ is at most $r/2$,
and so the induction hypothesis implies that $M_i \le 180^{\lfloor r/2 \rfloor + 1}$,
where $\lfloor - \rfloor$ is the floor function.
If $\rank J_0(M_2)(\Q) = 0$, then $M_2 \le 180 \le 180^{r/2}$, and thus $N \le 180^{r+1}$.
If $\rank J_0(M_2)(\Q) \ne 0$, then $\rank J_0(M_1)(\Q) \le -1 + r/2$ and hence $M_1 \le 180^{\lfloor r/2 \rfloor}$,
which shows that $N \le 180^{r+1}$.

Next assume that $N = p^e$ for a prime $p$ and for an integer $e \ge 2$.
In this case by \cref{lem:strictly<} \eqref{lem:strictly<:item1} the rank of $J_0(p)(\Q)$ does not exceed $r/e$.
If $r/e < 1$, then $\rank J_0(p)(\Q) = 0$, and hence $p \le 71$.
On the other hand by \cref{lem:strictly<} \eqref{lem:strictly<:item3}
we have $\rank J_0(p^{e-1})(\Q) \le r-1$, which implies $p^{e-1} \le 180^r$
by the induction hypothesis.
Thus in this case $N \le 180^{r+1}$.
If $r/e \ge 1$ then we have shown that in this case we obtain $p \le 180^{\lfloor r/e \rfloor}$ in the argument above.
Hence $N \le 180^r$.
This completes the proof.
\end{proof}

Note that, although the statements in this section treat the higher rank,
these, as we have remarked in preliminaries, assume only the converse of Kato's theorem \cite[Corollary 14.3]{Kato},
but not require the full strength of Birch--Swinnerton-Dyer conjecture.

\begin{acknowledgements}
I would like to thank my supervisor Takeshi Saito for his kind and valuable suggestions.
I also thank Yoichi Mieda for pointing out an error in an earlier version on the reference to Kato's theorem \cite{Kato}.
\end{acknowledgements}

\end{document}